\documentclass[11pt,a4paper]{amsart}

\usepackage[utf8]{inputenc}

\usepackage{bm} %boldmath anders als \mathbf (vor amsmath laden?)
\usepackage{amsfonts,amsmath,anysize,amsthm,stmaryrd}

\usepackage{graphicx,dsfont,fancyhdr}
\usepackage{enumerate}
\usepackage{verbatim}
\usepackage{eurosym} %fuer Euro-Zeichen z.B. \EUR{50}
\usepackage{natbib} %fuer schicke Quellenangaben
\usepackage{mathrsfs} %fuer Durantes Skript-J

\newcommand{\R}{\mathds{R}}

\newcommand{\N}{\mathds{N}}

\newcommand{\be}{\begin{enumerate}}
\newcommand{\ee}{\end{enumerate}}

\newcommand{\coloneqq}{\mathrel{\mathop:}=} %fuer huebsches :=
\newcommand{\eqqcolon}{=\mathrel{\mathop:}} %fuer huebsches =:

\newcommand{\BIGOP}[1]
{
\mathop{\mathchoice%
{\raise-0.22em\hbox{\huge $#1$}}%
{\raise-0.05em\hbox{\Large $#1$}}{\hbox{\large $#1$}}{#1}}} %fuer grosse Operatoren
\theoremstyle{plain}
\newtheorem{theorem}{Theorem}
\newtheorem{lemma}{Lemma}[section]
\newtheorem{corollary}{Corollary}[section]

\theoremstyle{definition}
\newtheorem{definition}{Definition}[section]
\newtheorem{example}{Example}[section]

\theoremstyle{remark}
\newtheorem{case}{Case}
\newtheorem{subcase}{Case}
\numberwithin{subcase}{case}

\begin{document}
%-----------------------------------HEADER-----------------------------------------------------------------------
%	\begin{center}
%		{\Large \textbf{Maximal Non-Exchangeability in Dimension \textit{d}}}
%		\vspace{0.25cm} \linebreak
%	\end{center}
%
%	\bigskip{}
\title{Maximal Non-Exchangeability in Dimension \textit{d}}
\author{Michael Harder \and Ulrich Stadtm\"uller}
\address{University of Ulm, Institute of Number Theory and Probability Theory, 89069 Ulm, Germany}
\email{michael.harder@uni-ulm.de}

\begin{abstract}
We give the maximal distance between a copula and itself when the argument is permuted for arbitrary dimension, generalizing a result for dimension two by \citet{nelsenExtrem,KlementMesiar}. Furthermore, we establish a subset of $[0,1]^d$ in which this bound might be attained. For each point in this subset we present a copula and a permutation, for which the distance in this point is maximal. In the process, we see that this subset depends on the dimension being even or odd.
\end{abstract}
\maketitle
\let\thefootnote\relax\footnote{Preprint submitted to Journal of Multivariate Analysis}
%-----------------------------------Inhalt-----------------------------------------------------------------------
\section{Introduction}
Studying the dependence structure in the distribution function $H$ of a $d$-dimensional continuous random vector $\bm{X}$ the so called copula is crucial. This is the distribution $C$ of the random vector $\bm{U}$ with components $U_i=F_i(X_i)$ where $F_i$ is the one-dimensional marginal distribution of $X_i\,.$ For details, see Sklar's Theorem in \cite{sklar59}.

Of interest are in particular parametric classes of such copulas. The usual examples, however, have the disadvantage that they share some symmetry properties. Quite popular are Archimedean copulas which have the form
$$ C(u_1,...,u_d)=\varphi (\varphi^{-1}(u_1)+ \dots,\varphi^{-1}(u_d))\,, $$
with a generating function $\varphi(s)$ being most often the Laplace transform of a distribution on $(0,\infty)$. If these generating functions contain some parameter $\theta$ we are given a parametric copula model. However, a random vector $U$ having this copula as a distribution has exchangeable components. But it is not clear whether data which have to be investigated follow an exchangeable copula. On the way to look for tests on exchangeability one comes across the question: what is the maximal distance between a copula and a version of it where the arguments are permuted. This paper is devoted to this question.

In the following, let $d \in \N\setminus\{1\}$ denote the dimension.
\begin{definition}\label{exchVecDef}
	A random vector $\bm{X}\coloneqq(X_1,\ldots,X_d)^\top$ is called \emph{exchangeable}, if its law coincides with the law of the random vector $\bm{X}_\pi \coloneqq (X_{\pi(1)},\ldots,X_{\pi(d)})^\top$, where $\pi \in S_d$ is a permutation of $\{1,\ldots,d\}$.
\end{definition}
Let $H$ be the cdf of $\bm{X}$ and $H_\pi$ the cdf of $\bm{X}_\pi$. Then it is straightforward to see, that if $\bm{X}$ is exchangeable, then all marginal cdfs must be identical.
\begin{definition}
	A mapping $F:\R^d \mapsto \R$ is called \emph{exchangeable}, if \begin{equation*}F(x_1,\ldots,x_d) = F(x_{\pi(1)},\ldots,x_{\pi(d)})\end{equation*} holds for all $(x_1,\ldots,x_d)^\top \in \R^d$ and all permutations $\pi \in S_d$.
\end{definition}
Note, that instead of \emph{exchangeable} the notion \emph{symmetric} is used as well (e.\,g.\ for aggregation functions by \citet{Grabisch}), which however is not used in a uniquely defined way (e.\,g.\ \citet{NelsenSymmetry} defines four different kinds of symmetry of a distribution function).
It may seem unusual to use the same word for a property of a random vector as well as for a property of a mapping. But it is easy to verify that a random vector is exchangeable if and only if its cdf is exchangeable. From the famous theorem by \citet{sklar59} it follows that a multivariate cumulative distribution function is exchangeable if and only if its copula is exchangeable (provided that all marginal cdfs are identical). In the following, we will address the exchangeability---or rather the lack of this property---of copulas.\\
 Now, being interested in statistical tests to decide whether some data come from an exchangeable copula it is important to know how big the difference of a copula from itself with permuted components can be. For exchangeable copulas this difference is zero. Here comes the first result in this direction.\\
\citet{nelsenExtrem} shows that for $d=2$ and any copula $C$ it holds that
\begin{equation}\label{nelsenEq}
	\lvert C(\bm{u}) - C(\bm{u}_\pi) \rvert \leq \frac{1}{3}\quad \mbox{ for all } \bm{u} \in [0,1]^2 \mbox{ and  all }\pi \in S_2\,.
\end{equation}
The same result has been published independently by \citet{KlementMesiar}. For $\pi = \mbox{id}$ obviously $C(\bm{u})=C(\bm{u}_\pi)$, so for $d=2$ there's only one interesting permutation, namely $\pi=\tau(1,2)$, i.e. the transposition of $u_1$ and $u_2$. The bound in \eqref{nelsenEq} is the best possible, as \citet{nelsenExtrem} demonstrates by showing that
\begin{equation*}
	 C(u_1,u_2)\coloneqq\min\biggl\{u_1,u_2,\Bigl(u_1-\frac{1}{3}\Bigr)^++\Bigl(u_2-\frac{2}{3}\Bigr)^+\biggr\}
\end{equation*}
is a copula and for $\bm{u}\coloneqq\bigl(\frac{1}{3},\frac{2}{3}\bigr)^\top$ the bound in \eqref{nelsenEq} is attained. As usual we denote by $f^+\coloneqq\max\{f,0\}\,.$

By defining $\tilde{C}(u_1,u_2) \coloneqq C(u_2,u_1)$ for any $(u_1,u_2)^\top \in [0,1]^2$, we obviously get another copula $\tilde{C}$. Therefore, \eqref{nelsenEq} could be rewritten as
\begin{equation}\label{nelsenEq2}
	\max_{\bm{u} \in [0,1]^2} \lvert C(\bm{u}) - \tilde{C}(\bm{u}) \rvert \leq \frac{1}{3}
\end{equation}
i.\,e.\ the maximal absolute difference between two copulas. However, the difference between two arbitrary 2-dimensional copulas in the same point is at most $0.5$, as
\begin{equation*}
	\lvert C_a(\bm{u}) - C_b(\bm{u}) \rvert \leq M(\bm{u}) - W(\bm{u}) \leq M\bigl(\tfrac{1}{2},\tfrac{1}{2}\bigr)-W\bigl(\tfrac{1}{2},\tfrac{1}{2}\bigr) = \frac{1}{2}
\end{equation*}
shows, where $M(u_1,u_2)\coloneqq\min\{u_1,u_2\}$ and $W(u_1,u_2)\coloneqq\max\{0,u_1+u_2-1\} $ are the upper and lower Fr\'{e}chet-Hoeffding-bounds, respectively. Note that this bound is best possible since it is attained by the two copulas $M$ and $W\,.$  Whereas the extension of the latter inequality to arbitrary dimension $d$ is obvious this is not the case for the inequality (\ref{nelsenEq}). Hence, it is aim of the present paper is to extend  inequality (\ref{nelsenEq}) to arbitrary dimension $d$ and to investigate the copulas and the set of points where this bound is attained.
%So we have to keep in mind, that as much as \eqref{nelsenEq2} isn't the difference between two arbitrary copulas at a point point, \eqref{nelsenEq}
%isn't the difference between one copula evaluated at two arbitrary points (in that case, for example $C(1,1) - C(0,0)= 1$ holds for any copula $C$). %Just as the points $u$ and $u_\pi$ in \eqref{nelsenEq} are not arbitrary (up to their order, they have the same components), the copulas $C$ and $\tilde{C}$ in \eqref{nelsenEq2} are not arbitrary either, but one determines the other.
%%
%%%%%%%%%% Section %%%%%%%%%%
%
\section{Main Result}\label{SecMainRes}
Now, let's state the main theorem of this paper, generalizing the inequality \eqref{nelsenEq} to arbitrary dimension $d$. Just like in Definition \ref{exchVecDef}, given a vector $\bm{u}\in [0,1]^d $,  we write   $\bm{u}_{\pi}\coloneqq(u_{\pi(1)}, \ldots, u_{\pi(d)})^\top$ for the vector whose components are permuted according to $\pi \in S_d$.
\begin{theorem}\label{mainTheo}	
	Let $C$ be a $d$-copula. Then
	\begin{equation}\label{mainTheoEq}
		\max_{\bm{u}\in[0,1]^d}\lvert C(\bm{u})- C(\bm{u}_\pi) \rvert \leq \frac{d-1}{d+1}
	\end{equation}
	holds true for any permutation $\pi \in S_d$.  The bound is best possible, i.e. for each dimension $d$ there exists a $d$-copula $C$, a permutation $\pi\in S_d$ and a vector $\bm{u}^\ast \in [0,1]^d$, such that $\lvert C(\bm{u}^\ast)- C(\bm{u}^\ast_\pi)\rvert =  \frac{d-1}{d+1}$.
\end{theorem}
\noindent{\bf Remarks:}{\it \\
i) The difference between two arbitrary copulas $C_1$ and $C_2$ of  dimension $d$ can be bounded for all $\bm{u} \in [0,\,1]^d $ as follows
\[ |C_1(\bm{u})-C_2(\bm{u})|\le M_d(\bm{u})-W_d(\bm{u})\le M_d(\bm{u}^*)-W_d(\bm{u^*})= \frac{d-1}{d}  \]
with the Fr\'{e}chet-Hoeffding-bounds $M_d(\bm{u})=\min\{u_1,\dots,u_d\}$ and %$W_d(\bm{u})=\max\{u_1+\dots +u_d-d+1,0\} $
$W_d(\bm{u})=\max\{\sum_{i=1}^du_i-d+1,0\} $, and $u_j^*=(d-1)/d$
for all $j=1,\dots,d\,.$ Although $W_d$ is no copula for $d>2$, the bound $\frac{d-1}{d}$ is best possible, since for every fixed $\bm{u} \in [0,1]^d$ there exists a copula $C$, such that $C(\bm{u})=W_d(\bm{u})$ (see e.\,g.\ \citet{nelsen} or for an exact form of such a copula with given diagonal section, see \citet{Jaworski}).\\[1mm]
ii) If we assume $u_1^\ast \leq u_2^\ast$,  \citet{nelsenExtrem} shows that for $d=2$ there is exactly one $\bm{u}^\ast=\bigl(\frac{1}{3},\frac{2}{3}\bigr)^\top$ for which the maximum in \eqref{mainTheoEq} is attained. %Assumed $u_1^\ast \leq u_2^\ast \leq u_3^\ast$ the same holds true for $d=3$ and  $\bm{u}^\ast=\bigl(\frac{1}{2},\frac{1}{2},1\bigr)^\top$. Still, there are two possible choices for $\bm{u}_\pi^\ast$: Either $\bm{u}_\pi^\ast=(u_3^\ast,u_1^\ast,u_2^\ast)^\top$ or $\bm{u}_\pi^\ast=(u_2^\ast,u_3^\ast,u_1^\ast)$. For $d =4$ we lose uniqueness, even if we assume $u_i^\ast \leq u_{i+1}^\ast$ ($i=1,\ldots,3$). Our general example (see \eqref{CopBsp}) works with $\bm{u}^\ast = (0.6,0.6,0.8,1)^\top$. There's another copula which is maximal non-exchangeable in $\bm{u}^\ast = (0.6,0.6,0.9,0.9)^\top$.\\[1mm]
Under the condition, that $u^\ast_1 \leq \cdots \leq u^\ast_d$, we get nonuniqueness or uniqueness of $\bm{u}^\ast$ %only for $d=2n+1, n\in\N$.
depending on $d$ being even or odd. For $d=2n+2, n \in \N$ there are infinitely many choices for such a $\bm{u}^\ast$---yet within some lower dimensional manifold. In any case, for $d>2$, a fixed $\bm{u}^\ast$ and a fixed copula $C$, such that the bound in \eqref{mainTheoEq} is achieved, there's still more than one choice for the permutation $\pi$. This will be discussed in more detail in Section \ref{SecAddRes}. \\[1mm]
iii) Based on our result we could define 
\begin{equation*}
	\mu(C) \coloneqq \frac{d+1}{d-1} \,\max\limits_{\pi \in S^d}\max\limits_{\bm{u} \in [0,1]^d} |C(\bm{u})-C(\bm{u}_\pi)|
\end{equation*} 
as a measure of non-exchangeability for the copula $C\,$. Note, that the definition of measures of non-exchangeability by \citet{Durante2010measures} is just for bivariate copulas and therefore not applicable in this case.
}
\vspace*{2mm}

In the following corollary we see that Theorem \ref{mainTheo} is not just a statement about exchangeability, but also has consequences for the possible choices of lower dimensional margins of a copula. For example, if $d>3$ there exists no copula, of which two $(d-1)$-dimensional margins $C_a$ and $C_b$ coincide on the point $\frac{d-2}{d-1}(1,\ldots,1)^\top$ with the Fr\'echet-Hoeffding-bounds.
\begin{corollary}\label{raenderCor}
	Let $d>3$, $C$ be a $d$-copula and $1\leq k<\frac{d-1}{2}$. Let $C_{(d-k),a}$ and $C_{(d-k),b}$ two $(d-k)$-dimensional margins of $C$. Then
	\begin{equation*}
		\lvert C_{(d-k),a}(\tilde{\bm{u}})-C_{(d-k),b}(\tilde{\bm{u}}) \rvert \leq \frac{d-1}{d+1} < \frac{d-k-1}{d-k} = M_{d-k}(\bm{u}^\ast)-W_{d-k}(\bm{u}^\ast)
	\end{equation*}
	for all $\tilde{\bm{u}} \in [0,1]^{d-k}$ and $\bm{u}^\ast \coloneqq \frac{d-k-1}{d-k}(1,\ldots,1)^\top\in[0,1]^{d-k}$
\end{corollary}
By $M_{d-k}$ we denote the upper $(d-k)$-dimensional Fr\'echet-Hoeffding-bound, and by $W_{d-k}$ a $(d-k)$-copula which coincides with the lower $(d-k)$-dimensional Fr\'echet-Hoeffding-bound in $\bm{u}^\ast$. Note, that Corollary \ref{raenderCor} is still correct for $d=3$, but gives no information.%Hier noch Referenz auf Sklar, dass so eine Copula exisitiert und evtl. dass die Schranke selbst keine Copula ist. -> steht weiter oben
\begin{proof}
	As $C_{(d-k),a}$  and $C_{(d-k),b}$ are margins of $C$, for a fixed $\tilde{\bm{u}} \in [0,1]^{d-k}$ there exist 
	%indices $i_1, \ldots, i_k$ and $j_1, \ldots, j_k$, such that
%	\begin{align*}
%		 C_{(d-k),a}(\tilde{\bm{u}})&=C(\overbrace{\tilde{u}_1,\ldots,\tilde{u}_{i_1-1},1,\tilde{u}_{i_1+1},\ldots}^{\eqqcolon\bm{u}_a}) \qquad \text{and}\\
%		 C_{(d-k),b}(\tilde{\bm{u}})&=C(\underbrace{\tilde{u}_1,\ldots,\tilde{u}_{j_1-1},1,\tilde{u}_{j_1+1},\ldots}_{\eqqcolon\bm{u}_b})
%	\end{align*}
%	hold. 
	$\bm{u}_a, \bm{u}_b \in [0,1]^d$ with exactly $k$ components equal to $1$, such that
	\begin{equation*}
		 C_{(d-k),a}(\tilde{\bm{u}}) = C(\bm{u}_a) \text{ and } C_{(d-k),b}(\tilde{\bm{u}}) = C(\bm{u}_b)\text{.}
	\end{equation*}
	These two $d$-dimensional vectors $\bm{u}_a$ and $\bm{u}_b$ are the same, up to the order of their components. Therefore, there exists a permutation $\pi \in S_d$ such that $\bm{u}_a=(\bm{u}_{b})_\pi$ and
	\begin{equation*}
		\lvert C_{(d-k),a}(\tilde{\bm{u}})-C_{(d-k),b}(\tilde{\bm{u}})\rvert =\lvert C(\bm{u}_a)-C\bigl((\bm{u}_a)_\pi\bigr)\rvert \leq \frac{d-1}{d+1} \text{.}
	\end{equation*}
	The other equations are straightforward to compute.
\end{proof}

%
%%%%%%%%%% Section %%%%%%%%%%
%
\section{Proof of the main result}
Before proving Theorem \ref{mainTheo} we first state some auxiliary results needed in the proof. By $\tau_{ij}$ we denote the transposition of $i$ and $j$,
i.e. the permutation interchanging components $i$ and $j$ and leaving the others unchanged.
\begin{lemma}\label{transpLemma}
	Let $\bm{u}\in [0,1]^d$, let $i,j \in \{1,\ldots,d\}$, then
	\begin{equation*}
		\lvert C(\bm{u}) - C(\bm{u}_{\tau_{ij}}) \rvert \leq \lvert u_i - u_j\rvert
	\end{equation*}
	holds for any $d$-copula $C$.
\end{lemma}
\begin{proof}
	Let $C$ be a $d$-copula, $\bm{u} \in [0,1]^d$ and $i,j \in \{1,\ldots,d\}$. %W.l.o.g.\ let $i=1$, $j=2$ and $u_i \leq u_j$.
	%W.\,l.\,o.\,g.\ we may assume $i<j$. 
	Now define $\bm{v}$ by
	\begin{equation*}
		v_k \coloneqq \max\{u_k, u_{\tau_{ij}(k)}\}, \quad k=1,\ldots,d
	\end{equation*}
	which implies $v_k=u_k$ for $k \ne i,j$. Due to the monotonicity of $C$ we get 
	\begin{equation}\label{transpEq1}
		C(\bm{u})\leq C(\bm{v}), \quad C(\bm{u}_{\tau_{ij}})\leq C(\bm{v})\text{.}
	\end{equation}
	$C$ being Lipschitz-continuous (see e.\,g.\ \citet{nelsen}) yields
	\begin{equation}\label{transpEq2}
		\lvert C(\bm{v}) - C(\bm{u}) \rvert \leq \sum_{k=1}^d \lvert v_k -u_k \rvert = \lvert v_i - u_i \rvert + \lvert v_j - u_j\rvert
	\end{equation}
	where the last equation is due to the choice of $\bm{v}$. As $v_i=v_j=\max\{u_i,u_j\}$ either $ \lvert v_i - u_i \rvert$ or  $\lvert v_j - u_j \rvert$ vanishes. Together with \eqref{transpEq1} we conclude 
	\begin{equation*}
		C(\bm{u}) \in \bigl[ C(\bm{v})-\lvert u_i-u_j \rvert, C(\bm{v}) \bigr]\text{.}
	\end{equation*}
	By replacing $\bm{u}$ in \eqref{transpEq2} by $\bm{u}_{\tau_{ij}}$, it is easy to see, that $C(\bm{u}_{\tau_{ij}})$ is within the same interval, which completes the proof.
\end{proof}

In the next lemma, we  will show that the upper inequality in Theorem \ref{mainTheo} holds. For the proof we need the following example of special permutations.

\begin{example}\label{permEx}
	Let $\bm{u}\in [0,1]^d$ and $\pi \in S_d$. Note that in this example each transposition might be the identity mapping. Let $\tau_d$ be the transposition, which exchanges $d$ and ${\pi(d)}$. Thus, $\tau_d$ puts $u_d$ in the right place. Now let $\tau_{d-1}$ be the transposition, which puts $u_{d-1}$ in $\bm{u}_{\tau_d}$ in the right place. If $(d-1)$ wasn't concerned by $\tau_d$ (i.e. $\tau_d(d-1) = d-1$), then $\tau_{d-1}$ is the transposition which exchanges $(d-1)$ and $\pi(d-1)$ (note that $\pi(d-1)\not=\pi(d)$, so $u_d$ remains untouched). Otherwise, $\tau_d(d) = d-1$ and then $\tau_{d-1}(d-1)=d-1$ and, even more important $\tau_{d-1}(d) = \pi(d-1)$. Now, we have $u_d$ and $u_{d-1}$ in the right places, i.e. on the same positions in $\bm{u}_\pi$ and $\bm{u}_{\tau_{d-1} \circ \tau_d}$. Like this, we can go on, until $\tau_2$ finally puts $u_2$ into its place. We needn't worry about $u_1$, because when $u_2, \ldots, u_d$ are all on their places, then $u_1$ has to be taken care of as well. In a nutshell, $\pi$ can be replaced by the composition of at most $d-1$ transpositions (for more details see e.\,g.\ \citet[p.\ 107]{DummitFoote}).
	
	Let's have a look at a concrete example, namely $\pi: (1,2,3,4) \mapsto (3,2,4,1)$. Now, one way to generate $\pi$ is by $\pi = \tau_2 \circ \tau_3 \circ \tau_4$, where the transpositions $\tau_j$ are characterized by
	\begin{equation*}
		\tau_4 = (34) \qquad  \tau_3=(14) \quad \text{ and } \quad \tau_2= \mathrm{id} \text{.}
	\end{equation*}
	In this case, as $\tau_2 = \mathrm{id}$, even two transpositions suffice to generate $\pi=(143)$.
\end{example}

\begin{lemma}\label{upperBndLemma}
	Let $\bm{u}\in [0,1]^d$, let $\pi \in S_d$, then
	\begin{equation}
		\lvert C(\bm{u}) - C(\bm{u}_{\pi}) \rvert \leq \frac{d-1}{d+1}
	\end{equation}
	holds for any $d$-copula $C$.
\end{lemma}
\begin{proof}
	Let $C$ be a $d$-copula. W.l.o.g.\ let $u_1\leq\ldots\leq u_d$, otherwise we replace $C$ in the proof by $\tilde{C}$ with $\tilde{C}(\bm{v})\coloneqq C(\bm{v}_{\sigma^{-1}})$ for all $\bm{v} \in [0,1]^d$. Here $\sigma \in S_d$ is the permutation which orders the components of $\bm{u}$ by size, i.e.\ $\bm{u}_\sigma = (u_{(1)},\ldots,u_{(d)})^\top$.
	
	If there exists at least one $i \in \{1,\ldots,d\}$ with $u_i <  \frac{d-1}{d+1}$ the claim follows immediately by
	\begin{equation*}
		\lvert C(\bm{u}) - C(\bm{u}_{\pi}) \rvert \leq \max\bigl\{ C(\bm{u}), C(\bm{u}_\pi)\bigr\} \leq M(\bm{u}) \le u_i <  \frac{d-1}{d+1} \text{.}
	\end{equation*}
	Hence we may assume now that $\frac{d-1}{d+1}\leq u_1$. In the following, we write $\tilde{u}_i \coloneqq u_i - \frac{d-1}{d+1}$, so we have $0 \leq \tilde{u}_i \leq \frac{2}{d+1}$. The permutation $\pi$ is generated by at most $(d-1)$ transpositions (as described in Example \ref{permEx}, see also \citet{DummitFoote}), therefore, we are able to write $\pi = \tau_2 \circ  \ldots \circ \tau_{d-1}  \circ \tau_d$. Next we use the triangular inequality to derive
	\begin{equation}\label{upperBndInEq}\begin{split}
		\lefteqn{ \lvert C(\bm{u}) - C(\bm{u}_{\pi}) \rvert \leq} \\ & \leq \lvert C(\bm{u}) - C(\bm{u}_{\tau_d}) \rvert + \lvert C(\bm{u}_{\tau_d}) - C(\bm{u}_{\tau_{d-1}\circ\tau_d}) \rvert + \ldots + \lvert C(\bm{u}_{\tau_3\circ\ldots\circ\tau_d}) - C(\bm{u}_{\pi})\rvert \\
		&\leq \sum_{i=2}^d (u_i - u_1) \leq \sum_{i=2}^d \tilde{u}_i
	\end{split}
\end{equation}
	where the second inequality follows from Lemma \ref{transpLemma}.
	
	At the same time, we have
	\begin{equation}\label{upperBndInEq2}
		\begin{split}
		\lvert C(\bm{u}) - C(\bm{u}_{\pi}) \rvert &\leq M_d(\bm{u})-W_d(\bm{u}) \\ & \leq u_1 - \Bigl(\sum_{i=1}^d u_i - (d-1)\Bigr) = 2\frac{d-1}{d+1} -\sum_{i=2}^d\tilde{u}_i
		\end{split}
	\end{equation}
	with the Fr\'echet-Hoeffding-bounds $M_d$ and $W_d$ (see \citet{nelsen}). %Hier vielleicht noch Referenz auf Original-Artikel.
	Therefore, we may conclude that
	\begin{equation*}
		\lvert C(\bm{u}) - C(\bm{u}_{\pi}) \rvert \leq \min\biggl\{ \sum_{i=2}^d \tilde{u}_i,2\frac{d-1}{d+1} -\sum_{i=2}^d\tilde{u}_i\biggr\} \leq \frac{d-1}{d+1}
	\end{equation*}
	which completes the proof.
\end{proof}

In the proof of Lemma \ref{upperBndLemma} we need $u_1\leq\ldots\leq u_d$ just for notational convenience. Therefore, it is straightforward to derive the following corollary:
\begin{corollary}
	With the prerequisites of Lemma \ref{upperBndLemma}% and $u_{(1)} \coloneqq \min\{u_1,\ldots,u_d\}$,
	\begin{equation*}
		\lvert C(\bm{u}) - C(\bm{u}_{\pi}) \rvert \leq \min\biggl\{u_1,\ldots,u_d, \sum_{i=1}^d (u_i-u_{(1)}), (d-1)+u_{(1)} -  \sum_{i=1}^d u_i \biggr\}
	\end{equation*}
	holds for any $d$-copula $C$ (where $u_{(1)} \coloneqq \min\{u_1,\ldots,u_d\}$).
\end{corollary}

By now, we established the upper inequality in Theorem \ref{mainTheo}. In order to prove that it cannot be improved, we have to find a proper $d$-copula, for which the bound in \eqref{mainTheoEq} is attained in some point $\bm{u} \in [0,1]^d$ and for some permutation $\pi \in S_d$. %We use the shuffle-of-min construction presented by \cite{Mikusinski1992} \cite{Mikusinski2009}.
To this end let $\bm{u}^\ast \in [0,1]^d$ such that
\begin{equation}\label{uStern}
	u_j^\ast \coloneqq \begin{cases}
		\frac{d-1}{d+1} & \text{for $1 \leq j \leq \frac{d+1}{2}$} \\
		%\frac{d-1}{d+1} &\text{for $j=\frac{d+1}{2}$ and $d$ odd} \\
		\frac{d}{d+1} &\text{for $j=\frac{d}{2}+1$ and $d$ even} \\
		1 & \text{otherwise}
	\end{cases}
\end{equation}
for $j \in \{1,\ldots,d\}$. In the following we consider the mapping $C^\ast : [0,1]^d \to \R$ with
\begin{equation}\label{CopBsp}
	C^\ast(\bm{u}) \coloneqq \sum_{j=0}^{d-1} \bigwedge_{k=0}^{d-1} \Bigl( u_{((j+k)\, \mathrm{mod}\,  d) + 1} - \!\!\!\!\!\!\sum_{i=1, i\not\in I(j,k)}^d \!\!\!\!\!\! (1-u_i^\ast) \Bigr)^+
\end{equation}
where $I(j,k) \coloneqq \bigl\{((j+l)\, \mathrm{mod}\,  d) + 1: l=0,1,\ldots,k\bigr\}$ and $\bigwedge_{i=1}^d a_i \coloneqq \min\{a_1,\ldots,a_d\}$.
A small calculation shows that in the case $d=2$ this copula satisfies $C^*(u_1,u_2)=\min\{u_1,u_2, (u_1-1/3)^++(u_2-2/3)^+\}$ as discussed above.
\begin{lemma}\label{BspLemma}
	Let $C^\ast$ be the mapping defined in \eqref{CopBsp} and $\bm{u}^\ast \in [0,1]^d$ as in \eqref{uStern}. Let $\pi \in S_d$ the order reversing permutation, i.e. $\pi(k)\coloneqq d-k+1$, then $C^\ast (\bm{u}^\ast)=0$ and $C^\ast (\bm{u}_\pi^\ast)=\frac{d-1}{d+1}$.
\end{lemma}
\begin{proof}
	First, note that $\sum_{i=1}^d u_i^\ast = d-1$ by the choice of $\bm{u}^\ast$ and therefore, $\sum_{i=1}^d (1- u_i^\ast) = 1$. Now for the first claim: let $j \in \{0,\ldots,d-1\}$ and $k\coloneqq0$. Thus $I(j,k)=I(j,0) = \{j+1\}$ and because of
	\begin{equation*}
		\sum_{i=1, i\not\in I(j,0)}^d \!\!\!\!\!\! (1-u_i^\ast) = \sum_{i=1}^d(1-u_i^\ast) - (1-u_{j+1}^\ast) = u_{j+1}^\ast
	\end{equation*}
	we get
	\begin{equation*}
		\Bigl( u_{((j+k)\, \mathrm{mod}\,  d) + 1}^\ast - \!\!\!\!\!\!\sum_{i=1, i\not\in I(j,k)}^d \!\!\!\!\!\! (1-u_i^\ast) \Bigr)^+ = (u_{j+1}^\ast-u_{j+1}^\ast)^+=0
	\end{equation*}
	whenever $k=0$. As this holds for each $j$ we have $C^\ast (\bm{u}^\ast)=0$.
	
	In order to prove the second claim, note that $C^\ast (\bm{u}^\ast)= \sum_{j=0}^{d-1} \bigwedge_{k=0}^{d-1} \bigl(m_{j,k}\bigr)^+$, with
	\begin{equation*}
		m_{j,k} \coloneqq   u_{([(d-1)-(j+k)]\, \mathrm{mod}\,  d) + 1}^\ast - \!\!\!\!\!\!\sum_{i=1, i\not\in I(j,k)}^d \!\!\!\!\!\! (1-u_i^\ast)
	\end{equation*}
	because $d-\bigl(((j+k)\, \mathrm{mod}\,  d) + 1\bigr)+1=\bigl((d-1) - (j+k)\bigr)\, \mathrm{mod}\, d +1$. Now let $j \in \{0,\ldots,d-1\}$ and $0\leq k\leq d-2$. We want to show that $m_{j,k}$ is nondecreasing in $k$, i.e. $m_{j,k}\leq m_{j,k+1}$. This is the case if and only if
	\begin{equation}\label{uSternEq}
		\begin{split}
			\alpha_{j,k} &\coloneqq u_{([(d-1)-(j+k)]\, \mathrm{mod}\,  d) + 1}^\ast - u_{([(d-1)-(j+k+1)]\, \mathrm{mod}\,  d) + 1}^\ast \\ 
			& \qquad \leq 1-u_{((j+k+1)\, \mathrm{mod}\,  d) + 1}^\ast \eqqcolon \beta_{j,k}
		\end{split}
	\end{equation}
	holds. Obviously the left hand side of \eqref{uSternEq} is the difference between consecutive components of $\bm{u}^\ast$, so $\alpha_{j,k}=0$ for most choices of $k$. The cases where $\alpha_{j,k}\not=0$ depend on $d$ being odd or even. If $d$ is even, $\alpha_{j,k}\not=0$ if:
	\begin{enumerate}
		\item\label{ersterFall} $([(d-1)-(j+k)]\, \mathrm{mod}\,  d) + 1=1$. Then $\alpha_{j,k}=u_1^\ast-u_d^\ast < 0 \leq \beta_{j,k}$.
		\item $([(d-1)-(j+k)]\, \mathrm{mod}\,  d) + 1=\frac{d}{2}+2$. In this case $((j+k+1)\, \mathrm{mod}\,  d) + 1=d-(\frac{d}{2}+1)+1=\frac{d}{2}$. Therefore,
		\begin{equation*}
			\alpha_{j,k}=u_{\frac{d}{2}+2}^\ast-u_{\frac{d}{2}+1}^\ast = \frac{1}{d+1} < \frac{2}{d+1} = 1-u_{\frac{d}{2}}^\ast = \beta_{j,k} \text{ .}
		\end{equation*}
		\item $([(d-1)-(j+k)]\, \mathrm{mod}\,  d) + 1=\frac{d}{2}+1$. In this case $((j+k+1)\, \mathrm{mod}\,  d) + 1=d-\frac{d}{2}+1=\frac{d}{2}+1$. Therefore,
		\begin{equation*}
			\alpha_{j,k}=u_{\frac{d}{2}+1}^\ast-u_{\frac{d}{2}}^\ast = \frac{1}{d+1}  = 1-u_{\frac{d}{2}+1}^\ast = \beta_{j,k} \text{ .}
		\end{equation*}
	\end{enumerate}
	If $d$ is odd, $\alpha_{j,k}\not=0$ if:
	\begin{enumerate}
		\item see \ref{ersterFall}. where $d$ is even.
		\item $([(d-1)-(j+k)]\, \mathrm{mod}\,  d) + 1=\frac{d+1}{2}+1$. In this case $((j+k+1)\, \mathrm{mod}\,  d) + 1=d-\frac{d+1}{2}+1=\frac{d+1}{2}$. Therefore,
		\begin{equation*}
			\alpha_{j,k}=u_{\frac{d+3}{2}}^\ast-u_{\frac{d+1}{2}}^\ast = \frac{2}{d+1} = 1- 1-u_{\frac{d+1}{2}}^\ast = \beta_{j,k} \text{ .}
		\end{equation*}
	\end{enumerate}
	So we have $\alpha_{j,k}\leq \beta_{j,k}$ and thus $m_{j,k}\leq m_{j,k+1}$ for all choices of $j$ and $k$. This means the minimum in \eqref{CopBsp} is always achieved for $k=0$ which gives us
	\begin{equation*}
		C^\ast(\bm{u}^\ast)= \sum_{j=0}^{d-1}(m_{j,0})^+ = \sum_{j=0}^{d-1} (u_{d-j}^\ast - u_{j+1}^\ast)^+=\frac{d-1}{d+1}
	\end{equation*}
	as for $j>\frac{d}{2}$ the term $(u_{d-j}^\ast - u_{j+1}^\ast)^+$ is $0$ by the construction of $\bm{u}^\ast$.
\end{proof}
Now we are finally set to prove Theorem \ref{mainTheo}.
\begin{proof}[Proof of Theorem \ref{mainTheo}]
	Let $\pi \in S_d$ and $C$ be a $d$-copula. Then by Lemma \ref{upperBndLemma} we get \eqref{mainTheoEq}. In Lemma \ref{BspLemma} we show that there exists a point $\bm{u}^\ast \in [0,1]^d$ and a mapping $C^\ast:[0,1]^d\to\R$ such that
	\begin{equation*}
		\lvert C^\ast(\bm{u}^\ast) - C^\ast(\bm{u}_\pi^\ast)\rvert = \frac{d-1}{d+1} \text{ .}
	\end{equation*}
	So, all we need to do in order to prove Theorem \ref{mainTheo} is to show that $C^\ast$ is indeed a copula. This is the case, as it can be constructed as a shuffle of min. In two dimensions \citet{mikusinski1992} show that by slicing the unit square vertically (including the mass of the upper Fr\'echet-Hoeffding-bound on the main diagonal) and rearranging it, i.e. shuffling the strips, the resulting mass distribution will yield a proper copula. \citet[section 6]{mikusinski2009} state that this also works for $d>2$ by rearranging $[0,1]^d$ (with the mass on $\{\bm{u} \in [0,1]^d \, | \, u_1=\ldots=u_d\}$). $[0,1]^d$ is separated along hyperplanes of the form $\{u_k = \lambda_k \}$. The separate parts are then rearranged. The resulting shuffle of the original mass distribution corresponds to a proper copula. $C^\ast$ can be obtained this way, by using hyperplanes with $\lambda_k \coloneqq \sum_{i=1}^k (1-u_i^\ast)$. \citet{Durante2010} generalize this concept by applying it to arbitrary copulas. By Remark 2.1. therein, and following their notation, we get a copula $\tilde{C}$ indicated by $\bigl\langle(\mathscr{J}^k)_{k=1}^d, (C_i)_{i=1}^d\bigr\rangle$ where $C_i(\bm{u})\coloneqq M_d(\bm{u})$ for $i=1,\ldots,d$,  and $\mathscr{J}^k = (J_j^k)_{j=1}^d$ with
	\begin{equation}\label{Fallunterscheidung}
		J_j^k \coloneqq \begin{cases}
			\bigl[\sum_{i=1, i\not=j,\ldots,k}^d (1-u_i^\ast) , \sum_{i=1, i\not=j+1,\ldots,k}^d (1-u_i^\ast)\bigr] & \text{if $j<k$, } \\
			\bigl[\sum_{i=1, i\not=k}^d (1-u_i^\ast) , 1\bigr] & \text{if $j=k$, } \\
			\bigl[\sum_{i=k+1}^{j-1} (1-u_i^\ast) , \sum_{i=k+1}^{j} (1-u_i^\ast)\bigr] & \text{if $j>k$, }
		\end{cases}
	\end{equation}
	for $k=1,\ldots,d$. In Proposition 2.2. \citet{Durante2010} give an explicit expression of $\tilde{C}$, namely
	\begin{equation}\label{duranteEq}
		\tilde{C}(\bm{u})=\sum_{j=1}^d \lambda(J_j^1)M_d\biggl(\frac{(u_1-a_j^1)^+}{\lambda(J_j^1)}  , \ldots,\frac{(u_d-a_j^d)^+}{\lambda(J_j^1)}  \biggr)
	\end{equation}
	where $a_j^k$ is the left limit of the interval $J_j^k$. Showing that $\tilde{C}(\bm{u})=C^\ast(\bm{u})$ is just notationally demanding. The sums in \eqref{CopBsp} and in \eqref{Fallunterscheidung} look similar, but in \eqref{CopBsp} we circumvent the distinction of cases by using modular arithmetic. Note that in \eqref{duranteEq}, we write $(u_i-a_j^i)^+$ instead of $u_i-a_j^i$ in Proposition 2.2. in \citet{Durante2010}. But from their proof it is clear that a summand is $0$ whenever $u_i<a_j^i$ for at least one $i \in \{1,\ldots,d\}$.
\end{proof}

%%%%%%%%%% Section %%%%%%%%%%
%
\section{Additional Results}\label{SecAddRes}
As mentioned in Section \ref{SecMainRes}, if we assume $u_1^\ast \leq u_2^\ast$,  \citet{nelsenExtrem} shows that for $d=2$ there is exactly one $\bm{u}^\ast$ (namely $\bm{u}^\ast=\bigl(\frac{1}{3},\frac{2}{3}\bigr)^\top$) for which the maximum in \eqref{mainTheoEq} is attained. For $d>2$, the point $\bm{u}^\ast$, where equality in \eqref{mainTheoEq} holds, is unique if and only if $d$ is odd (assumed $u_i^\ast \leq u_j^\ast$ for $i\leq j$). If $d=2n+2$ ($n\in\N$), then there is a $(\frac{d}{2}-1)$-dimensional manifold $\mathcal{M} \subset [0,1]^d$, such that for all $\bm{u}^\ast \in \mathcal{M}$, there exist a copula $C$ and a permutation $\pi\in S_d$ with $\lvert C(\bm{u}^\ast) - C(\bm{u}^\ast_\pi) \rvert = \frac{d-1}{d+1}$. This is shown in Lemma~\ref{oddLemma}. For the proof we are going to improve the bound in \eqref{upperBndInEq} which was derived in the proof of Lemma~\ref{upperBndLemma}.

\begin{lemma}\label{improvBoundLemma}
	Let $d\geq 2$ and $\bm{u} \in [0,1]^d$ with $u_i\leq u_j$ for $i\leq j$. Then
	\begin{equation}\label{imprvBndEq}
		\lvert C(\bm{u}) - C(\bm{u}_\pi) \rvert \leq \!\! \sum_{i = \left\lceil\frac{d}{2}\right\rceil+1}^d \!\! (u_i - u_1)
	\end{equation}
	holds for any copula $C$ and any permutation $\pi \in S_d$, where $\lceil a \rceil  $ denotes the smallest integer $n \ge a\,.$
\end{lemma}
Before the proof of Lemma \ref{improvBoundLemma} for an arbitrary $\pi$, we will give the proof for a special case in the following example.
\begin{example}\label{improvBoundExample}
	Let $d\geq 3$, $\bm{u}$ as in Lemma \ref{improvBoundLemma} and $\pi \in S_d$ such that $\pi(i) \not= i$ for exactly three $i \in \{1,\ldots,d\}$. This means, there are exactly three components $u_{i_1}, u_{ i_2}, u_{i_3} $ in $\bm{u}$, which are permuted in $\bm{u}_\pi$. W.\,l.\,o.\,g.\ we may assume $i_1<i_2<i_3$. As $\pi$ can't be a transposition (otherwise, there is one $k$ with $\pi(i_k)=i_k$), either $\pi$  is a left-shift or a right shift, i.\,e.\ $\pi=\pi_l\coloneqq(i_1i_3i_2)$ or $\pi=\pi_r\coloneqq(i_1i_2i_3)$ (as there are no other derangements in $S_3$). Now let $\tau_1 \coloneqq (i_1i_2)$ and $\tau_2 \coloneqq(i_2i_3)$, then $\pi_l$ and $\pi_r$ are generated by those two transpositions in the following way:
	\begin{equation*}
		\pi_l = \tau_1 \circ \tau_2, \quad \pi_r=\tau_2 \circ \tau_1\text{.}
	\end{equation*}
	So we have
	\begin{align*}
		\lvert C(\bm{u}) - C(\bm{u}_{\pi_l}) \rvert & \leq \lvert C(\bm{u}) - C(\bm{u}_{\tau_2}) \rvert+\lvert C(\bm{u}_{\tau_2}) - C(\bm{u}_{\pi_l}) \rvert \\
		\lvert C(\bm{u}) - C(\bm{u}_{\pi_r}) \rvert & \leq \lvert C(\bm{u}) - C(\bm{u}_{\tau_1}) \rvert+\lvert C(\bm{u}_{\tau_1}) - C(\bm{u}_{\pi_l}) \rvert 
	\end{align*}
	and applying Lemma \ref{transpLemma} yields
	\begin{align*}
		\lvert C(\bm{u}) - C(\bm{u}_{\pi_l}) \rvert & \leq \lvert u_{i_3} - u_{i_2} \rvert+ \lvert u_{i_2}-u_{i_1}\rvert  = \lvert u_{i_3} - u_{i_1} \rvert \leq \lvert u_d - u_1 \rvert \\
		\lvert C(\bm{u}) - C(\bm{u}_{\pi_r}) \rvert & \leq \lvert u_{i_2}-u_{i_1}\rvert + \lvert u_{i_3} - u_{i_2} \rvert = \lvert u_{i_3} - u_{i_1} \rvert \leq \lvert u_d - u_1 \rvert \text{.}
	\end{align*}
	Note that the last equation holds, as $u_1 \leq u_{i_1}\leq u_{i_2} \leq u_{i_3} \leq u_d$ by the prerequisites. Now, in this special case, \eqref{imprvBndEq} follows immediately, as either $\pi=\pi_l$ or $\pi=\pi_r$.
\end{example}
For more information on generating permutations by transpositions, see e.\,g.\ \citet{DummitFoote}. We will make use of Example \ref{improvBoundExample} in the following proof of Lemma \ref{improvBoundLemma}.
\begin{proof}
	Let $d\geq 2$, $\bm{u} \in [0,1]^d$ with $u_i\leq u_j$ for $i\leq j$ and $\pi \in S_d$. We will need $p\in \N$, defined by
	\begin{equation*}
		p \coloneqq \lvert \{1\leq i \leq d: \pi(i) \not= i\} \rvert
	\end{equation*}
	i.\,e.\ $p$ is the number of elements of $\{1,\ldots,d\}$, which are no fixed points of $\pi$. Note, that for $p=0$, there is nothing to show and $p=1$ is impossible. Therefore, we may assume $p\geq 2$ and have $p$ indices $1 \leq i_1 < \ldots <i_p \leq d$ with $\pi(i_k) \not= i_k$ for $k \in \{1,\ldots,p\}$.
	We will proof Lemma \ref{improvBoundLemma} by establishing the similar claim
	\begin{equation}\label{imprvBndEqP}
		\lvert C(\bm{u}) - C(\bm{u}_\pi) \rvert \leq \!\! \sum_{k = \left\lceil\frac{p}{2}\right\rceil+1}^p \!\! (u_{i_k} - u_{i_1})\text{.}
	\end{equation}
	Then \eqref{imprvBndEq} follows immediately, as 
	\begin{equation*}
		\sum_{k = \left\lceil\frac{p}{2}\right\rceil+1}^p \!\! (u_{i_k} - u_{i_1}) \leq \!\! \sum_{i = \left\lceil\frac{d}{2}\right\rceil+1}^d \!\! (u_i - u_1)
	\end{equation*}
	holds true for all $p$ and the corresponding index sets.
	
	The proof of \eqref{imprvBndEqP} will be an induction on $p$. For $p=2$ equation \eqref{imprvBndEqP} holds true due to Lemma \ref{transpLemma}. Now assume \eqref{imprvBndEqP} holds for $p-1$ (with $p\geq 3$). The proof will be completed by a case-by-case analysis, dependent on $y$ in $i_y \coloneqq \pi(i_p)$. In any case $y\not=p$ as $i_p$ is by definition no fixed point of $\pi$.
	\begin{case}\label{fall1}%Fall 1
		$y \in \left\{\left\lceil\frac{p}{2}\right\rceil+1,\ldots,p-1\right\}$: Just like in Example \ref{permEx}, we can see $\pi$ as a composition of at most $p-1$ transpositions, such that each $i_k$ is put in its place, starting with $i_p$. Therefore, we have $\pi=\sigma \circ \tau_p$, where $\tau_p \coloneqq \bigl(i_p\, \pi(i_p)\bigr)$ and $\sigma$ is the permutation which is generated by all the remaining transpositions. As $\tau_p(i_p)=\pi(i_p)$ by definition, $i_p$ is a fixed point of $\sigma$, so $\sigma$ permutes just $p-1$ elements. Thus we get
		\begin{align*}
			\lvert C(\bm{u})-C(\bm{u}_\pi)\rvert &\leq \lvert C(\bm{u})-C(\bm{u}_{\sigma})\rvert + \lvert C(\bm{u}_\sigma)-C(\bm{u}_\pi)\rvert \\
			&\leq  \!\! \sum_{k=\left\lceil \frac{p-1}{2} \right\rceil +1}^{p-1} \!\!\! (u_{i_k} - u_{i_1})+ (u_{i_p}-u_{i_y}) \leq  \!\! \sum_{k = \left\lceil\frac{p}{2}\right\rceil+1}^p \!\! (u_{i_k} - u_{i_1}) 
		\end{align*}
		by the induction hypothesis and Lemma \ref{transpLemma} as $\left\lceil\frac{p-1}{2}\right\rceil+1 \leq y \leq p-1$.		
	\end{case}
	\begin{case}%Fall 2
		$p=2n+1$ and $y=\left\lceil\frac{p}{2}\right\rceil=n+1$: Analogous to Case \ref{fall1}, as $\left\lceil\frac{p-1}{2}\right\rceil +1=n+1$.
	\end{case}
	\begin{case}\label{fall3} %Fall3
		$p = 2n$ and $y=\left\lceil\frac{p}{2}\right\rceil=n$: Now let $i_x \coloneqq \pi^{-1}(i_p)$ ($x\not=p$ as $i_p$ is not a fixed point of $\pi$).
		\begin{subcase}%Fall3.1
			$x>y$: Similar to Case \ref{fall1} (resp.\ Example \ref{permEx}) we write $\pi$ as a composition of tranpositions. This time $\pi = \sigma \circ \tau_1 \circ \tau_2$, with
			\begin{equation*}
				\tau_1 \coloneqq (i_x\, i_p), \quad \tau_2 \coloneqq (i_y\, i_x)
			\end{equation*}
			and $\sigma$ being the composition of the remaining transpositions. $i_p$ and $i_x$ are fixed points of $\sigma$, as $\tau_1 \circ \tau_2 (i_p) = \pi(i_p)$ and $\tau_1 \circ \tau_2(i_x)= \pi(i_x)$. So $\sigma$ permutes $p-2$ elements. Because of $\tau_1 \circ \tau_2=(i_y\,i_x\,i_p)$, with Example \ref{improvBoundExample} we get
			\begin{align*}
				\lvert C(\bm{u})-C(\bm{u}_\pi)\rvert &\leq\lvert C(\bm{u})-C(\bm{u}_{\sigma})\rvert + \lvert C(\bm{u}_\sigma)-C(\bm{u}_\pi)\rvert \\
				&\leq  \!\! \sum_{k=\left\lceil \frac{p-2}{2} \right\rceil +1,\,k\not=x}^{p-1} \!\!\! (u_{i_k} - u_{i_1})+ (u_{i_p}-u_{i_y}) \leq  \!\! \sum_{k = \left\lceil\frac{p}{2}\right\rceil+1}^p \!\! (u_{i_k} - u_{i_1}) 
			\end{align*}
			by the induction hypothesis and Lemma \ref{transpLemma}.
		\end{subcase}
		\begin{subcase}%Fall3.2
			$x=y$: With $\pi = \sigma \circ \tau$ and $\tau \coloneqq (i_x\, i_p)$ (see Case \ref{fall1} or Example~\ref{permEx}), we get \eqref{imprvBndEqP} analogous to Case \ref{fall1}.
		\end{subcase}
		\begin{subcase}%Fall3.3
			$x<y$: Similar to case 3.1.\ we write $\pi=\sigma \circ \tau_1 \circ \tau_2$. This time with
			\begin{equation*}
				\tau_1 \coloneqq (i_x\, i_y), \quad \tau_2 \coloneqq (i_y\, i_p)
			\end{equation*}
			we get \eqref{imprvBndEqP} analogous to Case 3.1.
		\end{subcase}
	\end{case}
	\begin{case}%Fall4
		$y \in \left\{1,\ldots,\left\lceil\frac{p}{2}\right\rceil-1\right\}$: With $i_x \coloneqq \pi^{-1}(i_p)$ this case can be solved analogous to Case \ref{fall3}, which completes the proof.	
		%		
		%fuer Diss, in Paper eher weglassen:
%		\begin{subcase}%Fall4.1
%			$x>y$: With a suitable $\sigma_2 \in S(1,\ldots,x-1,x+1,\ldots,d-1) \cong S_{d-2}$ we can split $\pi$ into $\sigma_2 \circ \sigma_1$, with $\sigma_1 = (yxd)$. Analogous to Case \ref{fall1}, we have
%			\begin{equation*}
%				\lvert C(\bm{u})-C(\bm{u}_\pi)\rvert \leq u_d-u_y +  \!\! \sum_{j=\left\lceil \frac{d-2}{2} \right\rceil +\eta, j\not=x}^{d-1} \!\!\! (u_j - u_1) \leq  \!\! \sum_{j = \left\lceil\frac{d}{2}\right\rceil+1}^d \!\! (u_j - u_1)\text{,}
%			\end{equation*}
%			where $\eta = 1 + \Ind\bigl(x \leq \left\lceil\frac{d}{2}\right\rceil\bigr)$.
%		\end{subcase}
%		\begin{subcase}%Fall4.2
%			$x=y$: With a suitable $\sigma \in S(1,\ldots,x-1,x+1,\ldots,d-1)\cong S_{d-2}$ we can split $\pi$ into $\sigma \circ \tau_{x,d}$. Analogous to Case \ref{fall1} we get \eqref{GleichungFall3.3}.
%		\end{subcase}
%		\begin{subcase}%Fall4.3
%			$x<y$: With a suitable $\sigma_2 \in S(1,\ldots,x-1,x+1,\ldots,d-1) \cong S_{d-2}$ we can split $\pi$ into $\sigma_2 \circ \sigma_1$, with $\sigma_1 = (xdy)$. Analogous to Case \ref{fall1}, we get \eqref{GleichungFall3.3}.
%		\end{subcase}
	\end{case}
\end{proof}

Now we are able to prove, that for $d>2$, the point $\bm{u}^\ast$, where maximal non-exchangeability is possible, is unique if and only if the dimension is odd.
\begin{lemma}\label{oddLemma}
	Let $d>2$, $\mathcal{C}_d \coloneqq \{C:[0,1]^d\to\R: C \text{ is a $d$-copula}\}$ and %$\mathcal{M}\coloneqq$ 
	\begin{equation*}
		\mathcal{M}\coloneqq \left\{\bm{u}\in[0,1]^d: u_1\leq\ldots\leq u_d, \exists \pi \in S_d\  \exists C \in \mathcal{C}_d \text{ s.t. }  \lvert C(\bm{u})-C(\bm{u}_\pi)\rvert = \tfrac{d-1}{d+1}\right\} \text{.}
	\end{equation*} %assumed $u_1 \leq \ldots \leq u_d$
	Then $\lvert \mathcal{M} \rvert=1$ if and only if $d=2n+1$ (for a $n\in\N$). If $d=2n$, then $\mathcal{M}$ is a $(n-1)$-dimensional manifold.
\end{lemma}
\begin{proof}
	Let  $\bm{u} \in \mathcal{M}$ and $\tilde{u}_i \coloneqq u_i - \frac{d-1}{d+1}\in \bigl[0,\, \frac{2}{d+1}\bigr]\ (*)$. %The lower inequality
	The left bound of $\tilde{u}_i$ follows from $\frac{d-1}{d+1}=|C(\bm{u})-C(\bm{u}_\pi)|\le M_d(\bm{u})\le u_i $ for any $i=1,\dots,d\,.$ From \eqref{upperBndInEq2} we find that any such $\bm{u}$ satisfies $2\frac{d-1}{d+1} -\sum_{i=2}^d\tilde{u}_i \geq \frac{d-1}{d+1}\,,$ i.e., it holds that
	\begin{equation*}%\label{oungl}
		\frac{d-1}{d+1} \geq \sum_{i=2}^d\tilde{u}_i\ge \sum_{i = \left\lceil\frac{d}{2}\right\rceil+1}^d\negthickspace \tilde{u}_i\,.
	\end{equation*}
 This and the inequality $\sum_{i = \left\lceil\frac{d}{2}\right\rceil+1}^d (u_i - u_1)=\sum_{i = \left\lceil\frac{d}{2}\right\rceil+1}^d (\tilde{u}_i - \tilde{u}_1) \geq \frac{d-1}{d+1}$ from Lemma \ref{improvBoundLemma} yield
	\begin{equation}\label{untenuSchlangeEq}
		\sum_{i = \left\lceil\frac{d}{2}\right\rceil+1}^d \negthickspace  \tilde{u}_i = \frac{d-1}{d+1}
	\end{equation}
	 for every $\bm{u}\in \mathcal{M}$. Let $d=2n+1$ then the only way for \eqref{untenuSchlangeEq} to be true is
	 \begin{equation*}
	 	\tilde{u}_1=\ldots=\tilde{u}_{\left\lceil\frac{d}{2}\right\rceil} = 0, \quad \tilde{u}_{\left\lceil\frac{d}{2}\right\rceil+1}=\ldots=\tilde{u}_d=\frac{2}{d+1}
	 \end{equation*}
	 as $0\leq \tilde{u}_j \leq \frac{2}{d+1}$ for all $j=1,\ldots,d$.
	
	 Now let $d=2n$ and $\bm{u} \in [0,1]^d$ with
	 \begin{equation*}
	 	u_1=\ldots =u_{n}= \frac{d-1}{d+1}, \quad u_{n+j} = \frac{d}{d+1}+\delta_j \text{ for } j=1,\ldots,n
	 \end{equation*}
	 such that $\delta_j \in \bigl[0,\frac{1}{d+1}\bigr]$ and
	 \begin{equation*}
	 	\delta_1\leq \ldots \leq \delta_n, \quad \sum_{j=1}^n \delta_j = \frac{n-1}{d+1}
	 \end{equation*}
	 holds. Let $\tilde{\mathcal{M}}$ be the set of all such $\bm{u}$. For each $\bm{u} \in \tilde{\mathcal{M}}$ there exists a permutation $\pi$ and a copula $C$, such that $\lvert C(\bm{u}) - C(\bm{u}_\pi)\rvert = \frac{d-1}{d+1}$. We will construct such a copula by the Shuffle of Min method, presented by \citet{mikusinski1992} and  \citet{Durante2010}, in %Appendix~\ref{Anhang}.
	 the Appendix. Therefore, we have $\tilde{\mathcal{M}}\subseteq \mathcal{M}$. Now, let $\bm{u}\in \mathcal{M}$. If we assume $u_{n+1}<\frac{d}{d+1}$, i.e., $\tilde{u}_{n+1} <\frac{1}{d+1}$ then equation \eqref{untenuSchlangeEq} implies that there exists some $\tilde{u}_{n+j}> \frac{2}{d+1}$ contradicting~($*$) in %front of equation \eqref{oungl}.
	 the beginning of the proof. Hence, we can write $u_{n+j}=\frac{d}{d+1} +\delta_j$ with $0\le \delta_1\le \dots \le \delta_n\,.$ Consequently we have $\tilde{u}_{n+j}= \frac{1}{d+1} +\delta_j\,, \;  j=1,\ldots,n$ and equation \eqref{untenuSchlangeEq} implies
	 \begin{equation*}
	 	\sum_{j=1}^n \delta_j = \frac{n-1}{d+1}
	 \end{equation*}
which means that $\bm{u} \in \tilde{\mathcal{M}}$ and thus $\mathcal{M}\subseteq\tilde{\mathcal{M}}$.
\end{proof}

The above proof shows, that for every $\bm{u} \in \mathcal{M}$ the first $\lceil \frac{d}{2} \rceil$ components are equal. Therefore, even for a fixed $\bm{u}^\ast \in \mathcal{M}$ and a fixed $C \in \mathcal{C}_d$ there's never a unique $\pi \in S_d$ which maximizes \eqref{mainTheoEq} (for $d>2$). E.\,g.\ let $\pi$ be such a permutation, then $\tilde \pi \coloneqq \pi \circ \tau_{12}$ maximizes \eqref{mainTheoEq} as well.

\section{Acknowledgements}
The authors wish to thank the anonymous referees for their suggestions, which---among others---led to a simpler and shorter proof of Lemma \ref{transpLemma}.
\appendix

%\include{Anhang}
%%%%%%%%%%%%%%%%%
%
%	Anhang (ohne include, da sonst erste Seite frei)
%
%%%%%%%%%%%%%%%%%
\section{Examples}\label{Anhang}
Let $d=2n$ and $\bm{u} \in \mathcal{M}$ as described in the proof of Lemma \ref{oddLemma}, i.e.
\begin{equation*}
	u_1=\ldots =u_{n}= \frac{d-1}{d+1}, \quad u_{n+j} = \frac{d}{d+1}+\delta_j \text{ for } j=1,\ldots,n
\end{equation*}
such that $\delta_j \in \bigl[0,\frac{1}{d+1}\bigr]$ and
\begin{equation*}
	\delta_1\leq \ldots \leq \delta_n, \quad \sum_{j=1}^n \delta_j = \frac{n-1}{d+1} \text{.}
\end{equation*}
Let $\pi \in S$ be the order reversing permutation, i.e.\ $\pi(j) = d-j+1$ for $j=1,\ldots,d$. By applying the Shuffle-Of-Min-Method, we will construct a copula $C$, such that $\lvert C(\bm{u}) - C(\bm{u}_\pi)\rvert = \frac{d-1}{d+1}$.

According to Remark 2.1.\ in \citet{Durante2010}, all that is needed for the construction of such a copula, is a so called \emph{shuffling structure} of $d$-dimensional orthotopes and a system of copulas $(C_i)$. We use $C_i \equiv M_d$ for all $i$ for simplicity, but other choices, especially non-singular copulas are possible. Now for the orthotopes $J_i^1\times\ldots \times J_i^d$ (with $i\in\{1,\ldots, 3n-1\}$): In the following, we will give $J_i^k$ for all cases of $i\in\{1,\ldots, 3n-1\}$ and $k \in\{1,\ldots,d\}$.
\setcounter{case}{0}
\begin{case}%Fall1
	$i \in \{1,\ldots,n-1\}$:
	\begin{subcase}
		$k \in \{1,\ldots,n-i\}\cup\{n+1,\ldots,2n\}$ then: \\ $J_i^k \coloneqq \bigl[\sum_{j=1}^{i-1}\bigl(\frac{1}{d+1}+\delta_j\bigr) , \sum_{j=1}^{i}\bigl(\frac{1}{d+1}+\delta_j\bigr) \bigl]$
	\end{subcase}
	\begin{subcase}
		$k = n-i+1$ then: $J_i^k \coloneqq \bigl[\frac{d-1}{d+1} , \frac{d}{d+1}+\delta_i \bigl]$
	\end{subcase}
	\begin{subcase}
		$i\geq 2$ and $k \in \{n-i+2,\ldots,n\}$ then: \\ $J_i^k \coloneqq \bigl[\sum_{j=1, j\not=n+1-k}^{i-1}\bigl(\frac{1}{d+1}+\delta_j\bigr) , \sum_{j=1, j\not=n+1-k}^{i}\bigl(\frac{1}{d+1}+\delta_j\bigr)\bigl]$
	\end{subcase}
\end{case}
\begin{case}%Fall2
	$i \in \{n,\ldots,2n-2\}$ and:
	\begin{subcase}
		$k =1$ then: \\[1mm]$J_i^k \coloneqq \bigl[\frac{d-2}{d+1}-\delta_n+\sum_{j=1}^{i-n}\bigl(\frac{1}{d+1}-\delta_j\bigr)  ,\frac{d-2}{d+1}-\delta_n+\sum_{j=1}^{i-n+1}\bigl(\frac{1}{d+1}-\delta_j\bigr)  \bigl]$
	\end{subcase}
	\begin{subcase}
		$n\geq 3$ and $k \in \{2,\ldots,2n-i-1\}$ then: \\[1mm] $J_i^k \coloneqq \bigl[\frac{d-3}{d+1}-\delta_n-\delta_{n+1-k}+\sum_{j=1}^{i-n}\bigl(\frac{1}{d+1}-\delta_j\bigr)   , \frac{d-3}{d+1}-\delta_n-\delta_{n+1-k}+\sum_{j=1}^{i-n}\bigl(\frac{1}{d+1}-\delta_j\bigr)  \bigl]$
	\end{subcase}
	\begin{subcase}
		$k=2n-i$ then: $J_i^k \coloneqq \bigl[\frac{d}{d+1}+\delta_{n+1-k}, 1\bigl]$
	\end{subcase}
	\begin{subcase}
		$i\geq n+1$ and $k \in \{2n-i+1,\ldots,n\}$ then: \\[1mm] $J_i^k \coloneqq \bigl[\frac{d-4}{d+1}-\delta_n+\sum_{j=1}^{i-n}\bigl(\frac{1}{d+1}-\delta_j\bigr)  , \frac{d-4}{d+1}-\delta_n+\sum_{j=1}^{i-n+1}\bigl(\frac{1}{d+1}-\delta_j\bigr)\bigl]$
	\end{subcase}
	\begin{subcase}
		$k \in \{n+1,\ldots,2n\}$ then: \\[1mm] $J_i^k \coloneqq \bigl[\frac{d}{d+1}-\delta_n+\sum_{j=1}^{i-n}\bigl(\frac{1}{d+1}-\delta_j\bigr)  , \frac{d}{d+1}-\delta_n+\sum_{j=1}^{i-n+1}\bigl(\frac{1}{d+1}-\delta_j\bigr)\bigl]$
	\end{subcase}
\end{case}
\begin{case}%Fall3
	$i \in \{2n-1,\ldots,3n-2\}$ and:
	\begin{subcase}
		$k =1$ then: \\[1mm] $J_i^k \coloneqq \bigl[\frac{d-2}{d+1}+\sum_{j=1}^{i-2n+1}\bigl(\frac{1}{d+1}-\delta_j\bigr) , \frac{d-2}{d+1}+\sum_{j=1}^{i-2n+2}\bigl(\frac{1}{d+1}-\delta_j\bigr) \bigl] $
	\end{subcase}
	\begin{subcase}
		$k \in \{2,\ldots,n\}$ then:\\[1mm] $J_i^k \coloneqq \bigl[\frac{d-4}{d+1}+\sum_{j=1}^{i-2n+1}\bigl(\frac{1}{d+1}-\delta_j\bigr) , \frac{d-4}{d+1}+\sum_{j=1}^{i-2n+2}\bigl(\frac{1}{d+1}-\delta_j\bigr)\bigl]$
	\end{subcase}
	\begin{subcase}
		$k \in \{n+1,\ldots,3n-2\}\setminus\{i-n+2\}$ then:\\[1mm] $J_i^k \coloneqq \bigl[\frac{d}{d+1}+\sum_{j=1,j\not=k-n}^{i-2n+1}\bigl(\frac{1}{d+1}-\delta_j\bigr) , \frac{d}{d+1}+\sum_{j=1,j\not=k-n}^{i-2n+2}\bigl(\frac{1}{d+1}-\delta_j\bigr)\bigl]$
	\end{subcase}
	\begin{subcase}
		$k = i-n+2$: $J_i^k \coloneqq \bigl[ \frac{d}{d+1} + \delta_{k-n},1\bigl]$
	\end{subcase}
\end{case}
\begin{case}%Fall4
	$i =3n-1$:
	\begin{subcase}
		$k =1$: $J_i^k \coloneqq \bigl[\frac{d-1}{d+1},1\bigl]$
	\end{subcase}
	\begin{subcase}
		$k \in \{2,\ldots,n\}$ then: $J_i^k \coloneqq \bigl[\frac{d-3}{d+1},\frac{d-1}{d+1}\bigl]$
	\end{subcase}
	\begin{subcase}
		$k \in \{n+1,\ldots,2n\}$ then: $J_i^k \coloneqq \bigl[\frac{d-2}{d+1}-\delta_n,\frac{d}{d+1}-\delta_n\bigl]$
	\end{subcase}
\end{case}
By Definition 2.1.\ in \citet{Durante2010}, the intervals $J_i^k$ must fulfill four conditions, in order to get a proper copula:
\begin{enumerate}%[(S1)]
	\item First, $i$ must run in a finite or countable index set. This is obviously the case, as $1\leq i \leq 3n-1$.
	\item Second, for every $k \in \{1,\ldots,d\}$ and $i_1 \not= i_2$ the intervals $J_{i_1}^k$ and $J_{i_2}^k$ have at most one endpoint in common. This condition is tedious to verify, but nonetheless fulfilled.
	\item Third, the orthotopes must be $d$-hypercubes, i.e. $\big\lvert J_{i}^{k_1}\big\rvert =\big\lvert J_{i}^{k_2}\big\rvert $ for every $i$ and every pair $k_1,k_2$. This is the case, as for every $k$
		\begin{equation*}
			\big\lvert J_i^k \big\rvert = \begin{cases}
				\tfrac{1}{d+1} + \delta_i &\text{ for $i \in \{1,\ldots,n-1\}$,} \\
				\tfrac{1}{d+1} + \delta_{i-n+1} &\text{ for $i \in \{n,\ldots,2n-2\}$,} \\
				\tfrac{1}{d+1} + \delta_{i-2n+2} &\text{ for $i \in \{2n-1,\ldots,3n-2\}$,} \\
				\tfrac{2}{d+1} &\text{ for $i=3n-1$.}
			\end{cases}
		\end{equation*}
	\item Last, for every $k$, the length of the intervals must sum up to $1$. %This is the case, as
		
		\begin{equation*}
			\sum_{i=1}^{3n-1} \big\lvert J_i^k\big\rvert = \sum_{i=1}^{n-1}\bigl(\tfrac{1}{d+1}+\delta_i \bigr) + \sum_{i=n}^{2n-2}\bigl(\tfrac{1}{d+1}-\delta_{i-n+1}\bigr) +\negthickspace \sum_{i=2n-1}^{3n-2}\negthickspace\bigl(\tfrac{1}{d+1}-\delta_{i-2n+2}\bigr) + \tfrac{2}{d+1} = 1
		\end{equation*}
		for every $k$.
\end{enumerate}
Analogous to \eqref{duranteEq} we get an explicit expression of $C$, namely
\begin{equation}\label{deltaBsp}
	C(\bm{u}) = \sum_{i=1}^{3n-1} \min\Bigl(\bigl(u_1-a_i^1\bigr)^+  , \ldots,\bigl(u_d-a_i^d\bigr)^+, \big\lvert J_i^1\big\rvert \Bigr)
\end{equation}
where $a_i^k$ is the left limit of the interval $J_i^k$. The distribution of mass within the $d$-hypercubes is arbitrary, as long as there is exactly the mass $\big\lvert J_i^1\big\rvert$ in the hypercube $J_i^1\times\ldots \times J_i^d$. In our example, all the mass is on the diagonal. For a non-singular copula, one could spread the mass evenly within the hypercubes, for example replace $M_d$ in \eqref{duranteEq} by the Independence Copula $\pi_d$.

Let's clarify things with two small examples for $d=4$:
\begin{example}\label{Bsp1}
	In \eqref{uStern} we get $\bm{u}^\ast = (0.6,0.6,0.8,1)$. The copula in \eqref{CopBsp} is given by
	\begin{eqnarray*}
		\lefteqn{C^\ast(\bm{u}) = \min\bigl\{(u_1-0.6)^+,(u_2-0.2)^+,u_3,u_4\bigr\}+}\\ &+& \min\bigl\{(u_2-0.6)^+, (u_3-0.4)^+, (u_4-0.4)^+, u_1 \bigr\}\\
		&+&\min\bigl\{ (u_3-0.8)^+, (u_4-0.8)^+, (u_1-0.4)^+, u_2\bigr\} + \min\bigl\{ \underbrace{(u_4-1)^+}_{=0}, \ldots\bigr\}\text{.}
	\end{eqnarray*}
	Therefore, we have $\lvert C^*(\bm{u}^\ast) - C^*(1,0.8,0.6,0.6) \rvert = 0.6$.
\end{example}
\begin{example}
	As the dimension is even, the point $\bm{u}^\ast$ in Example \ref{Bsp1} is not the only one, in which maximal non-exchangeability is achieved. Let $\delta_1 \in [0,0.1]$ and $\tilde{\bm{u}}=(0.6,0.6,0.8+\delta_1,1-\delta_1)$. Note that $1-\delta_1 = 0.8 + \delta_2$ if $\delta_1 + \delta_2 = 0.2$. The copula in \eqref{deltaBsp} is given by
	\begin{eqnarray*}
		\lefteqn{C(\bm{u}) = \min\bigl\{u_1,(u_2-0.6)^+,u_3,u_4,0.2+\delta_1 \bigr\}+} \\
			&+&\min\bigl\{ (u_1-0.2-\delta_1)^+, (u_2-0.8+\delta_1)^+, (u_3-0.6-\delta_1)^+,\\
&& \hspace*{4mm} (u_4-0.6-\delta_1)^+, 0.2+\delta_1\bigr\}  \\
			&+& \min\bigl\{(u_1-0.4)^+, u_2, (u_3-0.8-\delta_1)^+,(u_4-0.8)^+,0.2-\delta_1 \bigr\} \\
			&+&\min\bigl\{(u_1-0.6+\delta_1)^+, (u_2-0.2+\delta_1)^+, (u_3-0.8)^+, (u_4-1+\delta_1)^+,\delta_1 \bigr\}\\
			&+& \min\bigl\{ (u_1-0.6)^+, (u_2-0.2)^+, (u_3-0.2-\delta_1)^+, (u_4-0.2-\delta_1)^+, 0.4\bigr\} \text{.}
	\end{eqnarray*}
	Therefore, we have $\lvert C(\tilde{\bm{u}}) - C(1-\delta_1, 0.8+\delta_1, 0.6,0.6) \rvert = 0.6$. This copula $C$ is different from the copula $C^\ast$ in Example \ref{Bsp1}, as $C(\tilde{\bm{u}})=0$, but $C^\ast(\tilde{\bm{u}}) = \delta_1$.
\end{example}

%Bibliographie
%\bibliographystyle{alpha}
%\bibliographystyle{apalike}
%\nocite{Durante2009}
\bibliography{bibliography}
\bibliographystyle{apalike}
\end{document}